\documentclass{amsart}

\usepackage{color,graphicx,amssymb,latexsym,amsfonts,txfonts,amsmath,amsthm}
\usepackage{pdfsync,enumitem}

\usepackage{tikz}
\usetikzlibrary{matrix}

\usepackage{hyperref}
\hypersetup{
    colorlinks=true,       % false: boxed links; true: colored links
    linkcolor=blue,          % color of internal links
    citecolor=blue,        % color of links to bibliography
    filecolor=blue,      % color of file links
    urlcolor=blue           % color of external links
}

\newtheorem{theo}{Theorem}
\newtheorem{coro}{Corollary}
\newtheorem{prop}{Proposition}
\newtheorem{lemm}{Lemma}

\theoremstyle{remark}
\newtheorem{rema}{\bf Remark}

\begin{document}

\title{$(g,k)$-Fermat curves: an embedding of moduli spaces}

\author{Rub\'en A. Hidalgo}
\address{Departamento de Matem\'atica y Estad\'{\i}stica, Universidad de La Frontera. Temuco, Chile}
%{ORCID: 0000-0003-4070-2819} 
\email{ruben.hidalgo@ufrontera.cl}

\thanks{Partially supported by Projects Fondecyt 1230001 and 1220261}

\subjclass[2010]{Primary 30F10; 30F40; 14H37}
\keywords{Fuchsian groups, Riemann surfaces, Automorphisms, Algebraic curves}

%%%%%%%%%%%%%%%%%
%%%%%%%%%%%%%%%%%

\begin{abstract}
A closed Riemann surface $S$ is called a $(g,k)$-Fermat curve, where $g,k \geq 2$ are integers, if it admits a group $H \cong {\mathbb Z}_{k}^{2g}$ of conformal automorphisms acting freely and with $S/H$ of genus $g$. In this case, we say that $H$ is a $(g,k)$-generalized Fermat group of $S$. In this paper, we provide a description of $S$ in terms of fiber products and 
we study the uniqueness of $(g,k)$-Fermat groups.
\end{abstract}

\maketitle

%%%%%%%%%%%%%%%%%
%%%%%%%%%%%%%%%%%
\section{Introduction}
In 1890, Schwarz \cite{Schwarz} proved that the group of conformal automorphisms of a closed Riemann surface of genus $\gamma \geq 2$ is finite, and later, in 1893, Hurwitz \cite{Hurwitz} obtained that its order is at most $84(\gamma-1)$. The (conformal classes of) Riemann surfaces with non-trivial automorphisms define the branch locus ${\mathcal B}_{\gamma}$ of the moduli space ${\mathcal M}_{\gamma}$ (for $\gamma \geq 4$, ${\mathcal B}_{\gamma}$ corresponds to the topological singular locus). Since the nineteenth century, the study and classification of groups of conformal automorphisms of Riemann surfaces have attracted attention and it is still an active research topic.

Let $S$ be a closed Riemann surface of genus $\gamma \geq 2$ and let ${\rm Aut}(S)$ be its group of conformal automorphisms. If $g, k \geq 2$ are integers and ${\mathbb Z}_{k}^{2g} \cong H \leq {\rm Aut}(S)$ acts freely on $S$ with $X=S/H$ of genus $g$, then we say that $H$ is a {\it $(g,k)$-Fermat group} of $S$, that $S$ is a {\it $(g,k)$-Fermat curve} and that $(S,H)$ a {\it $(g,k)$-Fermat pair}. In this case, $\gamma=\gamma_{g,k}:=1+k^{2g}(g-1)$ (by the Riemann-Hurwitz formula) and $S$ is non-hyperelliptic (Proposition \ref{hipereliptico}). 
As there are many different pairs $(g,k)$ and $(\hat{g},\hat{k})$ such that $\gamma_{g,k}=\gamma_{\hat{g},\hat{k}}$ (for instance, $(g,k)=(2,8)$ and $(\hat{g},\hat{k})=(5,2)$; in which case $\gamma_{2,8}=\gamma_{5,2}=4097$), there is the possibility for the existence of Riemann surfaces which are simultaneously a $(g,k)$-Fermat curve and also a $(\hat{g},\hat{k})$-Fermat curve.

The $(g,k)$-Fermat curves (without that name) already appeared in \cite{Macbeath}, where Macbeath obtained infinitely many values $\gamma \geq 2$ such that there is a Riemann surface of such genus $\gamma$ with the maximal number $84(\gamma-1)$ of automorphisms (these surfaces are called Hurwitz's curves).

Let $(S,H)$ be a $(g,k)$-Fermat pair and set $X=S/H$ (a closed Riemann surface of genus $g$).
In Theorem \ref{teo:fiber}, we observe that $(S,H)$ can be described in terms of a fiber product of $2g$ copies of ${\mathbb Z}_{k}$-cyclic covers of $X$. 
A description in terms of Fuchsian groups is as follows. By the uniformization theorem, there is a torsion-free co-compact Fuchsian group $\Gamma$, acting on the hyperbolic plane ${\mathbb H}^{2}$, such that $X={\mathbb H}^{2}/\Gamma$. If we denote by $\Gamma'$ its derived subgroup, then $\widetilde{X}={\mathbb H}^{2}/\Gamma'$ is the homology cover of $X$; the highest abelian cover of $X$. The surface $\widetilde{X}$ is topologically the Loch Ness monster (an infinite genus orientable surface with exactly one end). The group ${\widetilde H}={\mathbb Z}^{2g} \cong \Gamma/\Gamma' < {\rm Aut}(\widetilde{X})$, called a homology group of $\widetilde{X}$, acts freely on $\widetilde{X}$ and $\widetilde{X}/\widetilde{H}=X$. 
Let $\Gamma^{(k)}$ be the subgroup generated by $\Gamma'$ and the Burnside $k$-kernel $\Gamma^{k}$ (i.e, by the $k$-powers of the elements of $\Gamma$). In this case, 
$\widetilde{X}_{k}={\mathbb H}^{2}/\Gamma^{(k)}$ is a closed Riemann surface admitting 
$H_{k}=\Gamma/\Gamma^{k} \cong {\mathbb Z}_{k}^{2g}$ as a group of conformal automorphisms, which acts freely on it, and such that $\widetilde{X}_{k}/H_{k}=X$. In particular, 
$(\widetilde{X}_{k},H_{k})$ is a $(g,k)$-Fermat pair. We note (see Proposition \ref{propo1}) that 
there is a biholomorphism $\psi:S \to \widetilde{X}_{k}$ such that $\psi H \psi^{-1}=H_{k}$  (for that reason, we also call 
$S$ the {\it $k$-homology cover} of $X$).

Let $\Gamma_{1}$ and $\Gamma_{2}$ be two torsion free co-compact Fuchsian groups  (which, in principle, are not assumed to be of the same genus).  
In 1986, Maskit \cite{Maskit:homology} proved that, if $\Gamma'_{1}=\Gamma'_{2}$, then $\Gamma_{1}=\Gamma_{2}$  (this fact may be thought of as a Kleinian groups version of Torelli's theorem \cite{andreotti}, see Section \ref{Torelli}). This result asserts, in particular, that the homology cover $\widetilde{X}$ determines $X$ and that it admits a unique group of conformal automorphisms $G \cong {\mathbb Z}^{2 \hat{g}}$, where $\hat{g} \geq 2$, which acts freely and with a quotient of genus $\hat{g}$ (in which case,  $\hat{g}=g$ and $G=\widetilde{H}$).

The equality $\Gamma_{1}^{(k)}=\Gamma_{2}^{(k)}$, for some $k$, ensures that both groups $\Gamma_{1}$ and $\Gamma_{2}$ have the same genus $g$ (since the genus of ${\mathbb H}^{2}/\Gamma_{j}^{(k)}$ is $1+k^{2g_{j}}(g_{j}-1)$, where $g_{j}$ is the genus of $\Gamma_{j}$).

Assume there are integers $2 \leq k_{1} < k_{2}<\cdots <k_{j} < \cdots$ such that $\Gamma_{1}^{(k_{j})}=\Gamma_{2}^{(k_{j})}$. 
Since the intersection of all these subgroups $\Gamma_{j}^{(k_{j})}$ is $\Gamma'_{j}$, 
it follows (as a consequence of the above Maskit's result) that $\Gamma_{1}=\Gamma_{2}$. 
Due to this, it seems natural to ask the following: 

\medskip
{\it Q1: Is there an integer $k_{g} \geq 2$ such that  
$\Gamma_{1}^{(k_{g})}=\Gamma_{2}^{(k_{g})}$ ensures $\Gamma_{1}=\Gamma_{2}$?}
\medskip

In terms of $(g,k)$-Fermat curves, the above question can be stated as follows. 

\medskip
{\it Q2: Is there an integer $k_{g} \geq 2$ such that, for any $k \geq k_{g}$, any $(g,k)$-Fermat curve admits a unique $(g,k)$-Fermat group? }

\medskip

A weaker form of the above i the following (see Section \ref{Sec:ejemplo}):

\medskip
{\it Q3: Is every $(g,k)$-Fermat group a normal subgroup?}
\medskip

In Theorem \ref{teo1} (see Theorem \ref{teo2} for the Fuchsian form), we obtain the following partial answers. 
\begin{enumerate}
\item If $p \geq 3$ is a prime integer, $r \geq 1$ and $g\notin \{1+ap+b(p-1)/2; a,b \in \{0,1,\ldots\}\}$, then any two $(g,p^{r})$-Fermat groups of $S$ are conjugated in ${\rm Aut}(S)$.

\item If either 
\begin{enumerate}
\item[(i)]
$k=p^{r}$, where $p>84(g-1)$ is a prime integer and $r\geq 1$, or 
\item[(ii)] $k=2$ and $S/H$ is hyperelliptic, 
\end{enumerate}
then any $(g,k)$-Fermat curve admits a unique $(g,k)$-Fermat group.
\end{enumerate}

In case (2)(ii) above, we are able to provide an explicit algebraic model for $S$ for which the action of $H \cong {\mathbb Z}_{2}^{2g}$ can be explicitly seen (see Section \ref{algebramodelo}). We don't have explicit models in the general situation, but we believe that our fiber product interpretation (Section \ref{fiberproduct}) may be used for constructing algebraic models for $(g,k)$-Fermat pairs \cite{HCurves}.

Finally, in Section \ref{Sec:moduli}, we apply the above results to obtain certain embeddings of moduli spaces. 
There is a natural holomorphic embedding $\Theta_{\Gamma^{(k)}}:{\mathcal T}(\Gamma) \hookrightarrow {\mathcal T}(\Gamma^{(k)})$, where ${\mathcal T}(\Gamma)$ and ${\mathcal T}(\Gamma^{(k)})$ are the Teichm\"uller spaces of $\Gamma$ and $\Gamma^{(k)}$, respectively. As $\Gamma^{(k)}$ is a characteristic subgroup of $\Gamma$, this holomorphic embedding induces a holomorphic map  (not necessarily one-to-one) $\Phi_{\Gamma^{(k)}}:{\mathcal M}(\Gamma) \to {\mathcal M}(\Gamma^{(k)})$ between the corresponding moduli spaces (which are complex orbifolds). In Proposition \ref{propo3}, we provide sufficient conditions for such a map to be injective. This condition is equivalent to the uniqueness, up to conjugation, of the $(g,k)$-Fermat groups. As a consequence (Proposition \ref{teo3}), for 
$p \geq 3$ be a prime integer such that $g\notin \{1+ap+b(p-1)/2; a,b \in \{0,1,\ldots\}\}$, the map $\Phi_{\Gamma^{(p^{r})}}:{\mathcal M}(\Gamma) \to {\mathcal M}(\Gamma^{(p^{r})})$ is injective.

 \medskip
 
 {\bf Notations:}
(a) If $G$ is a group and $A \subset G$, then we denote by $\ll A \gg$ the smallest normal subgroup of $G$ containing $A$.
(b) If $S$ is a Riemann surface and $H<{\rm Aut}(S)$, then we denote by ${\rm Aut}_{H}(S)$ the normalizer of $H$ in ${\rm Aut}(S)$. 
(c) Two pairs, $(S_{1},H_{1})$ and $(S_{2},H_{2})$, where $S_{j}$ is a Riemann surface and $H_{j}<{\rm Aut}(S_{j})$, are {\it isomorphic} (respectively, {\it topologically equivalent}) if there is an biholomorphism (respectively, an orientation preserving homeomorphism) $\psi:S_{1} \to S_{2}$ such that $\psi H_{1} \psi^{-1}=H_{2}$. 

%%%%%%%%%%%%%%%%%%%%%%%
%%%%%%%%%%%%%%%%%%%%%%%
\section{Some general properties on $(g,k)$-Fermat curves}
In this section, we recall some general properties on $(g,k)$-Fermat curves.

%%%%%%%%%%%%%%%%%
\subsection{Non-hyperellipticity}
Closed Riemann surfaces of genus $g \geq 2$ are classified into two classes: the hyperelliptic and non-hyperelliptic ones. The hyperelliptic ones are those admitting a conformal automorphism of order two with exactly $2(g+1)$ fixed points; called the hyperelliptic involution. The hyperelliptic involution is known to be central.

\begin{prop}\label{hipereliptico}
Every $(g,k)$-Fermat curve is non-hyperelliptic.
\end{prop}
\begin{proof}
Let $(S,H)$ be a given $(g,k)$-Fermat pair. Assume $S$ is hyperelliptic and let $\iota \in {\rm Aut}(S)$ be its hyperelliptic involution. As $\iota$ commutes with every element of $H$ and $H$ acts freely, then  its $2\gamma+2=2(1+k^{2g}(g-1))+2=2 k^{2g}(g-1)+4$ fixed points should be a multiple of $k^{2g}$, which is not possible.
\end{proof}

%%%%%%%%%%%%%%%%%%%%
\subsection{A Fuchsian description}
\begin{prop}\label{propo1}
Let $(S,H)$ be a $(g,k)$-Fermat pair, where $k,g \geq 2$, and let $\Gamma$ be a Fuchsian group such that $S/H={\mathbb H}^{2}/\Gamma$. Then
$(S,H)$ and $({\mathbb H}^{2}/\Gamma^{(k)},\Gamma/\Gamma^{(k)})$ are isomorphic. 
\end{prop}
\begin{proof}
 As $S$ is an unbranched Galois cover of $X$, there is a normal subgroup $F$ of $\Gamma$ such that $S={\mathbb H}^{2}/F$ and $H=\Gamma/F$. As $H$ is abelian, $\Gamma' \leq F$ and, as $H \cong {\mathbb Z}_{k}^{2g}$, $\Gamma^{k} \leq F$; so $\Gamma^{(k)} \leq F$. Since $\Gamma^{(k)}$ and $F$ both have index $k^{2g}$ in $\Gamma$, it follows that $F=\Gamma^{(k)}$. 
\end{proof}

\begin{coro}
Any two $(g,k)$-Fermat pairs, where $k,g \geq 2$, are topologically equivalent.
\end{coro}
\begin{proof}
If $\Gamma_{1}$ and $\Gamma_{2}$ are two co-compact torsion free Fuchsian groups of genus $g$, then there is a orientation-preserving homeomorphism 
$\varphi:{\mathbb H}^{2} \to {\mathbb H}^{2}$ such that  $\Gamma_{2}=\varphi \Gamma_{1} \varphi^{-1}$.
As $\Gamma_{j}^{(k)}$ is a characteristic subgroup of $\Gamma_{j}$,  also $\Gamma_{2}^{(k)}=\varphi \Gamma_{1}^{(k)} \varphi^{-1}$. The result now follows from Proposition \ref{propo1}.
\end{proof}

%%%%%%%%%%%%%%%%%%
\subsection{A universal property of $(g,k)$-Fermat curves}\label{Sec:Galois}
Let $(S,H)$ be a $(g,k)$-Fermat pair and let $\pi_{H}:S \to X=S/H$ be a Galois covering map with ${\rm deck}(\pi_{H})=H$. If $\Gamma$ is a Fuchsian group such that ${\mathbb H}^{2}/\Gamma=X$, then we may assume that $S={\mathbb H}^{2}/\Gamma^{(k)}$. There is a short exact sequence
$1 \to H \to {\rm Aut}_{H}(S) \stackrel{\rho}{\to} {\rm Aut}(X) \to 1$.
In particular, for every $A<{\rm Aut}(X)$ we may consider $\rho^{-1}(A)=\widetilde{A}<{\rm Aut}_{H}(S)$.

\begin{prop}\label{teo:cubriente}
Let $P:R \to X$ be an abelian Galois (unbranched) covering map with ${\rm deck}(P)=G$, where $G$ is a finite abelian group with exponent a divisor of $k$. Then
\begin{enumerate}[leftmargin=*,align=left]
\item There exists $L<H$ and a Galois covering map $\pi_{L}:S \to R$ with ${\rm deck}(\pi_{L})=L$, such that $\pi_{H}=P \circ \pi_{L}$, in particular, $G=H/L$. 

\item Let $A < {\rm Aut}(X)$ and $\pi:X \to X/A$ be a Galois (possible branched) cover with ${\rm deck}(\pi)=A$ and
consider the  (possible branched) covering map 
$\pi \circ P:S \to X/A$. Let $\rho^{-1}(A)=\widetilde{A} < {\rm Aut}(S)$ and let $K<H$ be the maximal $\widetilde{A}$-invariant subgroup of $H$ contained in $L$. If $Z=\widetilde{X}/K$, then $Q:Z \to X/A$, the branched covering induced by $\widetilde{A}/K$, is the closure Galois covering of $\pi \circ P$.

\item Assume $\tau \in {\rm Aut}(X)$ is an automorphism of prime order $p\geq 2$ which does not divides $k$, let  $\pi:X \to X/\langle \tau \rangle$ be a Galois (possible branched) cover with ${\rm deck}(\pi)=\langle \tau \rangle$ and let us consider the  (possible branched) covering map 
$\pi \circ P:R \to X/\langle \tau \rangle$. Then
{\rm (a)} There exists $\phi \in {\rm Aut}(S)$, of order $p$, such that $\langle \rho(\phi) \rangle=\langle \tau \rangle$.
{\rm (b)} The (branched) covering $\pi \circ P:S \to X/\langle \tau \rangle$ is a Galois covering if and only if $L$ is $\phi$-invariant under conjugation.
{\rm (c)} Let us assume that $X/\langle \tau \rangle$ has genus zero and $\tau$ has exactly $r \geq 3$ fixed points. If $K<N$ is invariant under conjugation by $\phi$, then $K \cong {\mathbb Z}_{k}^{s}$, for some $s \in \{0,1,\ldots,(p-1)(r-2)\}$ such that $k^{s} \equiv 1 \mod(p)$. In particular, in this case,  the Galois closure $Q:Z \to X/\langle \tau \rangle$ has deck group isomorphic to ${\mathbb Z}^{(p-1)(r-2)-s} \rtimes {\mathbb Z}_{p}$, where $s$ is maximum such that 
$L$ contains a $\phi$-invariant subgroup of $H$ being isomorphic to ${\mathbb Z}_{k}^{s}$.

\end{enumerate}
\end{prop}
\begin{proof}
(1) The abelian covering map $P:R \to X$ is determined by a surjective homomorphism $\theta:\Gamma \to G$ with kernel $\Gamma_{R}$ such that $R={\mathbb H}^{2}/\Gamma_{R}$ and $G=\Gamma/\Gamma_{R}$. As $G$ is abelian, $\Gamma' < \Gamma_{R}$ and, as every $k$-power of elements of $G$ is trivial, also $\Gamma^{k}<\Gamma_{R}$. So, $\Gamma^{(k)}<\Gamma_{R}$ and $L=\Gamma_{R}/\Gamma^{(k)}$.
(2) The Galois closure, in this case, corresponds to the subgroup $K=\cap_{a \in \widetilde{A}} \; aLa^{-1}$, which is maximal $\widetilde{A}$-invariant subgroup of $H$ contained in $L$.
(3) We know the existence of some $\eta \in {\rm Aut}(S)$ with $\rho(\eta)=\tau$. It follows that $\eta^{p} \in H$. If $\eta$ has order $p$, we take $\phi=\eta$. Otherwise, as $(\eta^{k})^{p}=(\eta^{p})^{k}=1$, and $(k,p)=1$, we have that $\rho(\eta^{k})$ must be non-trivial, and we may take $\phi=\eta^{k}$ (this takes care of (a)). Part (b) is a direct consequence of part (2). 
The first part of (c) follows from the existence of adapted homology basis for $X$ under $\tau$ due to Gilman \cite{Gilman} (see Remark \ref{Gilman} below). The second part is then consequence of part (a) and part (2).
\end{proof}

\begin{rema}[Gilman's adapted homology basis]\label{Gilman}
Let $X$ be a closed Riemann surface of genus at least two and $\tau \in {\rm Aut}(X)$ be a conformal automorphism of order a prime integer $p$ such that $X/\langle \tau \rangle$ has genus zero and exactly $r \geq 3$ cone points. Then there exists a basis $a_{1},\ldots,a_{2g}$ of $H_{1}(X;{\mathbb Z})$ (it might not be a canonical one) admitting a disjoint decomposition into $(r-2)$ sub-collections 
$\{a_{j_{1}},\ldots, a_{j_{p-1}}\}$, $j=1,\ldots, r-2$, such that, for each $j$ it holds that, if we set $a_{j_{p}}=(a_{j_{1}}a_{j_{2}}\cdots a_{j_{p-1}})^{-1}$ and $\tau_{*}$ is the induced action of $\tau$ on $H_{1}(X;{\mathbb Z})$, then
$\tau_{*}(a_{j_{1}})=a_{j_{2}}, \tau_{*}(a_{j_{2}})=a_{j_{3}}, \ldots, \tau_{*}(a_{j_{p-2}})=a_{j_{p-1}}, \tau_{*}(a_{j_{p-1}})=a_{j_{p}}, \tau_{*}(a_{j_{p}})=a_{j_{1}}.$
\end{rema}

%%%%%%%%%%%%%%%%%%%%%%%
\subsection{$(g,k)$-Fermat curves and Torelli's theorem}\label{Torelli}
Let $X={\mathbb H}^{2}/\Gamma$ be a closed Riemann surface of genus $g \geq 2$ and let $H^{1,0}(X) \cong {\mathbb C}^{g}$ be its space of holomorphic one-forms. The homology group $H_{1}(X;{\mathbb Z})$ is naturally embedded, as a lattice, in the dual space $(H^{1,0}(X))^{*}$ of $H^{1,0}(X)$ by integration of forms. The quotient $JX=(H^{1,0}(X))^{*}/H_{1}(X;{\mathbb Z})$ is a $g$-dimensional complex torus with a principally polarized structure obtained from the intersection form on homology. Torelli's theorem \cite{andreotti} asserts that $X$ is, up to biholomorphisms, determined by the principally polarized abelian variety $JX$. Let $\pi:(H^{1,0}(X))^{*} \to JX$ be a  holomorphic Galois cover induced by the action of $H_{1}(X;{\mathbb Z})$. If we fix a point $p \in X$, then there is a natural holomorphic embedding  $\varphi:X \hookrightarrow JX: q \mapsto \left[\int_{p}^{q}\right]$. It holds that (i) $\pi^{-1}(\varphi(X))=\widetilde{X}$ is a Riemann surface admitting the group $H_{1}(X;{\mathbb Z})$ as a group of conformal automorphisms such that $X=\widetilde{X}/H_{1}(X;{\mathbb Z})$ and (ii) $\widetilde{X}={\mathbb H}^{2}/\Gamma'$ ($\widetilde{X}$ is homeomorphic to the Loch Ness monster, i.e., the infinite genus surface with exactly one end). In this way, Torelli's theorem is ``in some sense'' equivalent to the commutator rigidity for $\Gamma$.
If $\alpha_{1},\ldots,\alpha_{g}, \beta_{1},\ldots, \beta_{g}$ is a basis for $H_{1}(X;{\mathbb Z})$, then $\langle \alpha_{1}^{k},\ldots,\alpha_{g}^{k}, \beta_{1}^{k},\ldots, \beta_{g}^{k}\rangle$ is a basis for $H_{1}(X;{\mathbb Z})^{(k)}$ (the subgroup of $H_{1}(X;{\mathbb Z})$ generated by the $k$-powers of all its elements). The quotient $g$-dimensional torus $J_{k}X=(H^{1,0}(X))^{*}/H_{1}(X;{\mathbb Z})^{(k)}$ has as induced polarization the $k$-times the principal one and it admits a group $H=H_{1}(X;{\mathbb Z})/H_{1}(X;{\mathbb Z})^{(k)}\cong {\mathbb Z}_{k}^{2g}$ of automorphisms such that $JX=J_{k}X/H$. There is a natural isomorphism between $JX$ and $J_{k}X$ preserving the polarizations (amplification by $k$). In particular, $X$ is uniquely determined (up to isomorphisms) by $J_{k}X$. Let $\pi_{k}:(H^{1,0}(X))^{*} \to J_{k}X$ be a holomorphic Galois cover induced by the action of $H_{1}(X;{\mathbb Z})^{(k)}$.
If $S=\pi_{k}(\widetilde{X}) \subset J_{k}X$, then $(S,H)$ is a $(g,k)$-Fermat pair with $X=S/H$ and $S={\mathbb H}^{2}/\Gamma^{(k)}$. In this way, the uniqueness of $(g,k)$-Fermat groups, up to conjugacy,  is somehow related to the determination of $X$, up to isomorphisms, by the abelian variety $J_{k}X$.

\subsubsection{\bf The prime case}
Let us assume $k=p$ is a prime integer and let $(S,H)$ be a $(g,p)$-Fermat pair and set $X=S/H$.
There is associated to the covering $S \to X$ the Prym variety $P(S/H)$, which satisfies that $JS$ is isogenous to the product $JX \times P(S/X)$.

If $L \cong {\mathbb Z}_{p}^{2g-1}$ is a subgroup of $H$, then $Y=S/L$ is a closed Riemann surface of genus $p(g-1)+1$ and there is 
a Galois unbranched covering $Y \to X$ with deck group $H/L \cong {\mathbb Z}_{p}$. Associated to this covering is the Prym variety $P(Y/X)$, which satisfies that $JY$ is isogenous to the product $JX \times P(Y/X)$. 

Let us recall that  the number of these maximal  subgroups $L$ of $H$ is $m_{p,g} = \dfrac{p^{2g}-1}{p-1}$.

\begin{theo}[\cite{CHR}]
Let $(S,H)$ be a $(g,p)$-Fermat pair, where $p$ is a prime integer and set $X=S/H$.
If $\mathcal{L} =  \{L_1 , \ldots , L_{m_{p,g}}\}$ is the set of  maximal  subgroups of $H$ and $Y_{j}=S/L_{j}$, then 	
$$
 P(S/X) \cong_{\textup{isog}}  \prod_{j=1}^{m_{q,g}} P(Y_j/X).
$$

\end{theo}

%%%%%%%%%%%%%%%%%
%%%%%%%%%%%%%%%%%
%%%%%%%%%%%%%%%%%
\section{$(g,k)$-Fermat curves as fiber products}\label{fiberproduct}
In this section, we provide a fiber product description of a $(g,k)$-Fermat pair $(S,H)$. The first description is given in terms of certain maximal subgroups of $H$. The second one, is provided in ternsm of ceratin subgroups of the Fuchsian group uniformizing the quotient $S/H$.

%%%%%%%%%%%%%%%%
\subsection{\bf Fiber product property in terms of $H$}
\begin{theo}\label{teo:fiber}
Let $(S,H)$ be a $(g,k)$-Fermat pair and $\pi:S \to X$ be a Galois covering, with deck group $H$.
If $H=\langle a_{1},\ldots,a_{2g}\rangle \cong {\mathbb Z}_{k}^{2g}$, then let $K_{j} \cong {\mathbb Z}_{k}^{2g-1}$ be generetad by $\{a_{1},\ldots, a_{j-1},a_{j+1},\ldots, a_{2g}\}$.
Set $Y_{j}=S/K_{j}$ and let  $Q_{j}:S \to Y_{j}$ be a Galois cover with deck group $K_{j}$. Let $P_{j}:Y_{j} \to X=S/H$ be a cyclic covering, with deck group $H/K_{j} \cong {\mathbb Z}_{k}$, such that $\pi=P_{j} \circ Q_{j}$.
Then $(S,H)$ is the fiber product of the $2g$ pairs $(Y_{j}=S/K_{j},P_{j}:Y_{j} \to X)$, $j=1,\ldots,2g$.
\end{theo}
\begin{proof}
The fiber product, provided by the $2g$ pairs $(Y_{j}=S/K_{j},P_{j}:Y_{j} \to X)$, where $j=1,\ldots,2g$, is given by the following one-dimensional compact analytic space):
$$Z=\{(y_{1},\ldots,y_{2g}) \in Y_{1} \times \cdots \times Y_{2g}: P_{1}(y_{1})=P_{2}(y_{2})=\cdots=P_{2g}(y_{2g})\} \subset Y_{1} \times \cdots \times Y_{2g}.$$

Let $\pi_{j}:Z \to Y_{j}$, defined by $\pi_{j}(y_{1},\ldots,y_{2g})=y_{j}$, and let 
$P:Z \to X$ the (analytic) map defined by $P_{1} \circ \pi_{1} = \cdots = P_{2g} \circ \pi_{2g}$.

As each of the $P_{j}$ has no branch values, it follows from the implicit function theorem, that $Z$ is a finite union of closed Riemann surfaces and, moreover, 
$P:Z \to X$ is a Galois covering whose deck group is $J:=C_{1} \times \cdots \times C_{2g} \cong {\mathbb Z}_{k}^{2g}$, where $C_{j}=H/K_{j}$. 

The connected components of $Z$ are known to be pairwise isomorphic closed Riemann surfaces \cite{HCV}. 
Let $Z_{1}$ be one of the connected componts of $Z$ and let $J_{1}$ be its $J$-stabilizer.  Let us consider the restrictions of $\pi_{j}:Z_{j} \to Y_{j}$ and $P:Z_{1} \to X$. Note that $P$ is a Galois covering with deck group $J_{1}$.

The universal property of fiber products asserts that there is a holomorphic (covering) map $\varphi:S \to Z_{1}$ such that $Q_{j}=\pi_{j} \circ \varphi$ \cite{HCV} (so $\pi=P_{j} \circ Q_{j}=P \circ \varphi$).

The subgroups $K_{j}$ satisfy that, for $1 \leq j_{1}<j_{2}<\cdots<j_{l} \leq 2g$, one has that $K_{j_{1}}\cap \cdots \cap K_{j_{l}} \cong {\mathbb Z}_{k}^{2g-l}$. Again, by the universal property of the fiber product, $Z_{1}$ is obtained as the quotient of $S$ by the full intersection $K_{1} \cap \cdots \cap K_{2g}$, which is the trivial group, it follows that $\varphi$ is injective. Now, as $\pi=P \circ \varphi$, the map $P:Z_{1} \to X$ has degree $k^{2g}$, so $J_{1}=J$. In particular, $Z=Z_{1}$.
This means that $(Z,J)$ is a $(g,k)$-Fermat pair such that $Z/J=S/H$ and that there is an isomorphism $\varphi:S \to Z$ such that $\varphi H \varphi^{-1}=J$.
\end{proof}

\subsection{\bf Fiber product property in terms of Fuchsian groups}
In the above, we started from a $(g,k)$-Fermat pair and we observed that it can be described as a fiber product. In the next, we provide the same result from the point of view of the Fuchsian group uniformizing the quotient surface $X$.
Let $\Gamma$ be a torsion free co-compact  Fuchsian group such that $X={\mathbb H}^{2}/\Gamma$. 
Let us consider a canonical set of generators for $\Gamma$, say
$$\Gamma=\langle \alpha_{1},\alpha_{2},\cdots,\alpha_{2g-1},\alpha_{2g}: [\alpha_{1},\alpha_{2}]\cdots [\alpha_{2g-1},\alpha_{2g}]=1 \rangle,$$
and the following $2g$ index $k$ normal subgroups 
$$
\Gamma_{j}=\ll \alpha_{1},\ldots, \alpha_{j-1}, \alpha_{j}^{k},\alpha_{j+1},\ldots,\alpha_{2g}\gg, \; j=1,\ldots, 2g,
$$
where $\ll A \gg$ denotes the smallest normal subgroup of $\Gamma$ containing $A \subset \Gamma$. Note that $K=\cap_{j=1}^{2g} \Gamma_{j}=\Gamma^{(k)}$, which is a normal subgroup of index $k^{2g}$ of $\Gamma$. Set $Y_{j}={\mathbb H}^{2}/\Gamma_{j}$ (a closed Riemann surface of genus four) and $P_{j}:Y_{j} \to X$ be the induced Galois cover given by the inclusion of $\Gamma_{j} \lhd \Gamma$ and whose deck group is $G_{j}=\Gamma/\Gamma_{j}=\langle \varphi_{j} \rangle \cong {\mathbb Z}_{k}$. The fiber product provided by the pairs $(Y_{1},P_{1}),\ldots,(Y_{2g},P_{2g})$ is given by
$$Z=\{(y_{1},\ldots,y_{2g}) \in Y_{1} \times \cdots \times Y_{2g}: P_{1}(y_{1})=\cdots= P_{j}(y_{j})=\cdots =P_{2g}(y_{2g})\}.$$

As a consequence of the implicit function theorem, it can be seen that $Z$ is a compact Riemann surface 
(which it might be disconnected).  Any of these connected components is isomorphic to ${\mathbb H}^{2}/K$ \cite{HRV}. As 
$K$ has index $k^{2g}$ in $\Gamma$, we obtain that $Z$ is a connected Riemann surface.
On $Z$ there is the group $H=\langle a_{1},\ldots,a_{2g} \rangle \cong {\mathbb Z}_{k}^{2g}$ of conformal automorphisms where
$$a_{j}(x_{1},\ldots,x_{4})=(x_{1},\ldots,x_{j-1},\varphi_{j}(x_{j}),x_{j+1},\ldots,x_{2g}).$$

There are natural Galois covers
$$\pi_{j}:Z \to Y_{j}: (y_{1},\ldots,y_{2g}) \mapsto y_{j},$$
whose deck group $K_{j} \cong {\mathbb Z}_{k}^{2g-1}$, generated by the set $\{a_{1},\ldots,a_{2g} \} \setminus \{a_{j}\}$, and
$$P:Z \to X: (y_{1},\ldots,y_{2g}) \mapsto P_{1}(y_{1}),$$
whose deck group is $H$. By the definitions, $P=P_{j} \circ \pi_{j}$, for each $j=1,\ldots,2g$.
It follows that $(Z,H)$ is a $(g,k)$-Fermat pair with $Z/H=X$. 

Summarizing the above construction is the following.

\begin{theo}\label{fibrado}
Let $X$ be a genus $g \geq 2$ Riemann surface and let $\Gamma$ be a torsion-free co-compact Fuchsian group such that $X={\mathbb H}^{2}/\Gamma$. Let 
$\Gamma_{1},\ldots,\Gamma_{2g}$ be the above pairwise distinct index $k$ normal subgroups  of $\Gamma$. Then the associated $(g,k)$-Fermat pair $(S,H)$ of $X=S/H$ is given as 
the fiber product of the $2g$ pairs 
$({\mathbb H}^{2}/\Gamma_{1},\Gamma/\Gamma_{1}),\ldots, ({\mathbb H}^{2}/\Gamma_{2g},\Gamma/\Gamma_{2g})$.
\end{theo}

%%%%%%%%%%%%%%%%%%
%%%%%%%%%%%%%%%%%
\section{Uniqueness of $(g,k)$-Fermat groups}
In the foloowing, we observe the uniqueness of the $(g,k)$-Fermat groups for the case of  $k=p^{r}$, where $r \geq 1$ and $p$ is a prime integer.

\begin{theo}\label{teo1}
Let $(S,H)$ be a $(g,p^{r})$-Fermat pair, where $p$ is a prime integer and $r \geq 1$. Then
\begin{enumerate}[leftmargin=*,align=left]
\item If $p>84(g-1)$, then $H$ is the unique $(g,p^{r})$-Fermat group of $S$. 
\item If $r=1$ and $p \geq 3$ is a prime integer such that $g\notin \{1+ap+b(p-1)/2; a,b \in \{0,1,\ldots\}\}$, then $H$ is the unique, up to conjugation, $(g,p)$-Fermat group of $S$.
\item If $r=1$, $p=2$ and $S/H$ is hyperelliptic, then $H$ is the unique $(g,2)$-Fermat group of $S$.
\end{enumerate}
\end{theo}
\begin{proof}
Let $(S,H)$ be a $(g,k)$-Fermat pair, let $\gamma=\gamma_{g,k}$ be the genus of $S$, and let  $\Gamma$ be a Fuchsian group such that $X=S/H={\mathbb H}^{2}/\Gamma$.

%%%%%%%%%%%%%%%%%%
\subsubsection*{\bf Part (1).}
Let us assume $k=p^{r}$, where $p$ is a prime integer such that  $p > 84(g-1)$. Let $L \geq 1$ be such that $|{\rm Aut}(S)|=|{\rm Aut}_{H}(S)|L$.
 As $X$ admits no conformal automorphisms of order $p$ (by the Riemann-Hurwitz formula, every automorphism of prime order of $X$ must have order at most $2g+1$), it follows that $H$ is a $p$-Sylow subgroup of ${\rm Aut}(S)$. Let $n_{p}$ be the number of 
$p$-Sylow subgroups of ${\rm Aut}(S)$. By  Sylow's theorem, $n_{p}$ is congruent to one module $p$ and divides $|{\rm Aut}(S)|=|{\rm Aut}_{H}(S)|L$. 
 
 As seen in Section \ref{coro0}, $|{\rm Aut}_{H}(S)| = p^{2rg} |{\rm Aut}(X)|$, so $n_{p}$ must divide $|{\rm Aut}(X)|L$. Since
 $|{\rm Aut}(X)|p^{2rg}L=|{\rm Aut}_{H}(S)|L=|{\rm Aut}(S)| \leq 84(\gamma-1)=84 p^{2rg}(g-1)$ (the inequality is the Hurwitz upper bound and the last equality is due to the fact that the genus of $S$ is $\gamma=1+p^{2rg}(g-1)$), we have that $|{\rm Aut}(X)|L \leq 84(g-1)$, so $n_{p} \leq 84(g-1)$.
 If $n_{p}>1$, that is, $n_{p}=1+sp$, some $s \geq 1$, then $p<1+sp \leq 84(g-1)$, a contradiction.

%%%%%%%%%%%%%%%%%%
\subsubsection*{\bf Part (2).}
If $p \geq 3$, then the condition $g\notin \{1+ap+b(p-1)/2; a,b \in \{0,1,\ldots\}\}$ asserts (by the Riemann-Hurwitz formula) that there is no conformal automorphism of order $p$ on a closed Riemann surface of genus $g$. It follows from Sylow's theorem that $H$ is a $p$-Sylow subgroup.

%%%%%%%%%%%%%%%%%%%
\subsubsection*{\bf Part (3).}
Let $k=2$ and $X$ be hyperelliptic, with hyperelliptic involution $\iota$. Let $K$ be a Fuchsian group acting on the hyperbolic plane ${\mathbb H}^{2}$ such that ${\mathbb H}^{2}/K=X/\langle \iota \rangle$ (the Riemann sphere with exactly $2g+2$ cone points of order two). The group $K$ has a presentation of the form
$K=\langle y_{1},\ldots,y_{2g+2}:y_{1}^{2}=\cdots=y_{2g+2}=y_{1}y_{2}\cdots y_{2g+2}=1\rangle.$
The (unique) index two torsion-free subgroup of $K$ is  
$\Gamma^{*}=\langle y_{1}y_{2},\ldots, y_{1}y_{2g+2}\rangle.$
In this case, $X={\mathbb H}^{2}/\Gamma^{*}$ (the hyperelliptic involution $\iota$ is induced by each of the generators $y_{i}$); so we may set $\Gamma=\Gamma^{*}$.
We claim that $K'=\Gamma^{2}$. In fact, as (i) $\Gamma^{2}$ is a characteristic subgroup of $\Gamma$ and (ii) $\Gamma$ is a normal subgroup of $K$, it follows that $\Gamma^{2}$ is a normal subgroup of $K$. As each of the commutators $[y_{i},y_{j}]=y_{i}y_{j}y_{i}^{-1}y_{j}^{-1}=(y_{i}y_{j})^{2} \in \Gamma^{2}$, we observe that $K'$ is a subgroup of $\Gamma^{2}$. Since $[K:\Gamma^{2}]=[K:\Gamma][\Gamma:\Gamma^{2}]=2 \times 2^{2g}=2^{2g+1}$ and $[K:K']=2^{2g+1}$, it follows the desired equality.
In this way, $S={\mathbb H}^{2}/K'$ is a generalized Fermat curve of type $(2,2g+1)$ whose generalized Fermat group of the same type is $K/K' \cong {\mathbb Z}_{2}^{2g+1}$ (see \cite{GHL} for details on generalized Fermat curves). The generalized Fermat group $K/K'$ is generated by involutions $a_{1}, \ldots, a_{2g+1}$, where $a_{j}$ is induced by the generator $y_{j}$. We set by $a_{2g+2}$ the one induced by $y_{2g+2}$, so $a_{1}\cdots a_{2g+2}=1$. It is known that the only elements of $K/K'$ acting with fixed points on $S$ are the elements $a_{j}$ (see \cite{GHL}). This permits to 
note that $H$ is the unique index two subgroup of $K/K'$ acting freely on $S$, this being $H=\langle a_{1}a_{2},a_{1}a_{3},\ldots, a_{1}a_{2g+2}\rangle$.

Now, let us assume there is another $(g,2)$-Fermat group $L$ of $S$. If $L$ is a subgroup of $K/K'$, then $L=H$ (by the uniqueness of the index two subgroups of $K/K'$ acting freely on $S$). So, let us assume that there is some $\alpha \in L-H$. As $K/K'$ is the unique generalized Fermat group of type $(2,2g+1)$ of $S$ \cite{HKLP}, $\alpha$ normalizes it. As $H$ is its unique index two subgroup acting freely on $S$, $\alpha$ also normalizes $H$. 
As $\alpha$ has order two, and it normalizes $K/K'$, it induces a M\"obius transformation $\beta$ of order two that permutes the $2g+2$ cone points of $S/(K/K')=\widehat{\mathbb C}$. There are two possibilities: (A) none of the cone points is fixed by $\beta$, or (B) $\beta$ fixes exactly two of them. Up to post-composition by a suitable M\"obius transformation, we may assume these cone points to be $\infty, 0, 1, \lambda_{1},\ldots, \lambda_{2g-1}$ and that in case (A) $\beta(\infty)=0$, $\beta(1)=\lambda_{1}$ and $\beta(\lambda_{2j+1})=\lambda_{2j}$ ($j=1,\ldots, g-1$) and that in case (B) $\beta(\infty)=\infty$, $\beta(0)=0$, $\beta(1)=\lambda_{1}$, and $\beta(\lambda_{2j+1})=\lambda_{2j}$ ($j=1,\ldots, g-1$). Note that in case (A) $\beta(z)=\lambda_{1}/z$ and in case  (B) we must have $\lambda_{1}=-1$ and $\beta(z)=-z$.
In \cite{GHL} it was proved that $S$ can be represented by an algebraic curve  of the form
\begin{equation}\label{curva}
\left\{ \begin{array}{ccc}
x_{1}^{2}+x_{2}^{2}+x_{3}^{2}&=&0\\
\lambda_{1}x_{1}^{2}+x_{2}^{2}+x_{4}^{2}&=&0\\
\vdots & \vdots& \vdots\\
\lambda_{2g-1}x_{1}^{2}+x_{2}^{2}+x_{2g+2}^{2}&=&0
\end{array}
\right\} \subset {\mathbb P}^{2g+1}
\end{equation}
and, in this model, $a_{j}([x_{1}:\cdots:x_{2g+2}])=[x_{1}:\cdots:x_{j-1}:-x_{j}:x_{j+1}:\cdots:x_{2g+2}]$.

Assume we are in case (A). Following Corollary 9 in \cite{GHL},  
$\alpha([x_{1}:\cdots:x_{2g+2}])=$
$$=[x_{2}:A_{2}x_{1}:A_{3}x_{4}:A_{4}x_{3}:\cdots: A_{2j-1}x_{2j}:A_{2j}x_{2j-1}:\cdots:A_{2g+1}x_{2g+2}:A_{2g+2}x_{2g+1}],$$
where 
$A_{2}^{2}=\lambda_{1}, \; A_{3}^{2}=1, \; A_{2j-1}^{2}=\lambda_{2j-4}, \; A_{2j}^{2}=\lambda_{2j-3}.$
As $\alpha$ has order two, we must also have
$A_{2}=A_{3}A_{4}=A_{5}A_{6}=\cdots=A_{2j-1}A_{2j}=\cdots=A_{2g+1}A_{2g+2}.$

The point $[1:\mu:p_{3}:\cdots:p_{2g+2}]$, where
$\mu^{2}=A_{2}$, $p_{3}=\sqrt{(\lambda_{1}-1)/(1-\mu^{2})}$, $p_{4}=\mu p_{3}/A_{3}$, $p_{2j}=\mu p_{2j-1}/A_{2j-1}$ and 
$p_{2j-1}=\sqrt{(\lambda_{2j-3}-\lambda_{2j-4})/(1-A_{2j}/A_{2j-1})}$, 
is a fixed point of $\alpha$ in $S$ (in the above algebraic model). This is a contradiction to the fact that $\alpha$ must act freely on $S$.

Assume we are in case (B). Again, in this case $\alpha([x_{1}:\cdots:x_{2g+2}])=
[x_{1}:A_{2}x_{2}:$ $A_{3}x_{4}:A_{4}x_{3}:\cdots: A_{2j-1}x_{2j}:A_{2j}x_{2j-1}:\cdots:A_{2g+1}x_{2g+2}:A_{2g+2}x_{2g+1}],$
where, for every $j$, $A_{j}^{2}=-1.$ In this case,
$\alpha^{2}([x_{1}:\cdots:x_{2g+2}])=[x_{1}:-x_{2}:x_{3}:\cdots:x_{2g+2}],$
which is a contradiction for $\alpha$ to be an involution.
\end{proof}

\begin{coro}
Let $p \geq 2$ be a prime integer and $S$ a $(2,p)$-Fermat curve. Then
\begin{enumerate}
\item If $p=2$ or $p \geq 87$, then $S$ has a unique $(2,p)$-Fermat group.
\item If $3 \leq p \leq 83$, then any two $(2,p)$-Fermat groups of $S$ are conjugated.
\end{enumerate}
\end{coro}
\begin{proof}
Case (1) and case (2) for $p \geq 7$ is direct consequence of Theorem \ref{teo1}. Case (2) for $p \in \{3,5\}$ follows from the fact that a Riemann surface of genus two admitting an automorphism of such order $p$ is uniquely determined up to biholomorphisms.
\end{proof}

In terms of Fuchsian groups, Theorem \ref{teo1}  may be stated as follows.

\begin{theo}\label{teo2}
Let $\Gamma_{1}$ and $\Gamma_{2}$ be torsion free co-compact Fuchsian groups. Let $g \geq 2$ be the genus of $S={\mathbb H}^{2}/\Gamma_{1}$, $p \geq 2$ be a prime integer and $r \geq1$.
\begin{enumerate}[leftmargin=*,align=left]
\item If $p>84(g-1)$ and $\Gamma_{1}^{(p^{r})}=\Gamma_{2}^{(p^{r})}$, then  $\Gamma_{1}=\Gamma_{2}$.
\item If $p \geq 3$ is such that $g\notin \{1+ap+b(p-1)/2; a,b \in \{0,1,\ldots\}\}$ and  $\Gamma_{1}^{(p^{r})}=\Gamma_{2}^{(p^{r})}$, then
$\Gamma_{1}$ and $\Gamma_{2}$ are ${\rm Aut}({\mathbb H}^{2})$-conjugated.
\item If $p=2$, $S$ is hyperelliptic and $\Gamma_{1}^{(2)}=\Gamma_{2}^{(2)}$, then  $\Gamma_{1}=\Gamma_{2}$.
\end{enumerate}
\end{theo}

\subsection{A remark on the group ${\rm Aut}(S)$ for $S$ a $(g,k)$-Fermat curve}\label{coro0}
Let $(S,H)$ be a $(g,k)$-Fermat pair, where $k,g \geq 2$, and let $\Gamma$ be a Fuchsian group such that $S/H={\mathbb H}^{2}/\Gamma$.

\noindent
(1) As $\Gamma^{(k)}$ is a characteristic subgroup of $\Gamma$, there is a short exact sequence
$$1 \to H \to {\rm Aut}_{H}(S) \to {\rm Aut}(X) \to 1,$$ 
where ${\rm Aut}_{H}(S)$ is the normalizer of $H$ in ${\rm Aut}(S)$.
In particular (this was used by Macbeath \cite{Macbeath} in order to construct infinitely many Hurwitz curves), 
$$|{\rm Aut}(S)| \geq |{\rm Aut}_{H}(S)|=|{\rm Aut}(X)|k^{2g}=\left(\frac{|{\rm Aut}(X)|}{g-1}\right)(\gamma-1).$$ 

\noindent
(2) Part (1) of Theorem \ref{teo1} asserts that, if $k=p^{r}$, where $p>84(g-1)$ is a prime integer and $r \geq 1$, then $H$ is the unique 
$(g,p^{r})$-Fermat group. In particular, ${\rm Aut}(S)={\rm Aut}_{H}(S)$ and $|{\rm Aut}(S)|=p^{2rg}|{\rm Aut}(X)|$. 
If $g \geq 3$, then (in the generic situation) one has that ${\rm Aut}(S/H)$ is trivial, in which case, ${\rm Aut}(S)=H$.
For $g=2$, we have that $|{\rm Aut}(S)|=|{\rm Aut}(X)|p^{4r}\leq 48 p^{4r}$ (equality holds when $X$ is the closed Riemann surface of genus two with exactly $48$ automorphisms: $y^{2}=x(x^{4}-1)$).

\subsection{Example}\label{Sec:ejemplo}
Let us consider a $(g,k)$-Fermat pair $(S,H)$, set $X=S/H$ and assume that $S/{\rm Aut}(X)$ has triangular signature $(0;r,s,t)$. In this case, as $S/{\rm Aut}_{H}(S)=X/{\rm Aut}(X)$, if we assume this triangular signature to be finitely maximal \cite{Singerman} (i.e., a triangular Fuchsian group of type $(r,s,t)$ cannot be estrictly contained as a finite index subgroup of other Fuchsian group), then necessarily ${\rm Aut}_{H}(S)={\rm Aut}(S)$ and, in particular, $H$ is a normal subgroup of ${\rm Aut}(S)$.

\subsection{On $2$-homology covers of hyperelliptic surfaces}\label{algebramodelo}
Let $(S,H)$ be a $(g,2)$-Fermat pair such that $S/H$ hyperelliptic defined by 
$y^{2}=x(x-1)\prod_{j=1}^{2g-1}(x-\lambda_{j})$. We have seen, in the proof of part (3) of Theorem \ref{teo1}, that $S$ is the generalized Fermat curve of type $(2,2g+1)$ defined by
\begin{equation}
S:= \left\{ \begin{array}{ccc}
x_{1}^{2}+x_{2}^{2}+x_{3}^{2}&=&0\\
\lambda_{1}x_{1}^{2}+x_{2}^{2}+x_{4}^{2}&=&0\\
\vdots & \vdots& \vdots\\
\lambda_{2g-1}x_{1}^{2}+x_{2}^{2}+x_{2g+2}^{2}&=&0
\end{array}
\right\} \subset {\mathbb P}^{2g+1}.
\end{equation}

It could be of interest in describing algebraic curves for $S$, for any type $(g,k)$, in terms of the algebraic equations for $S/H$. This will be pursued elsewhere (an strategy is to use the fiber product description in Section \ref{fiberproduct}).

\begin{coro}\label{2-homology}
Two hyperelliptic Riemann surfaces are isomorphic if and only if their corresponding $2$-homology covers are.
\end{coro}
\begin{proof}
If $X_{j}={\mathbb H}^{2}/\Gamma_{j}$, then $S_{j}={\mathbb H}^{2}/\Gamma_{j}^{2}$. One direction is clear, if $X_{1}$ and $X_{2}$ are isomorphic, then $\Gamma_{1}$ and $\Gamma_{2}$ are conjugated by some element of ${\rm PSL}_{2}({\mathbb R})$. Such a conjugation preserves the characteristic subgroups, that it also conjugates $\Gamma_{1}^{2}$ and $\Gamma_{2}^{2}$. In the other direction, without lost of generality, we may assume that $\Gamma_{1}^{2}=\Gamma^{2}_{2}$. In particular, $S_{1}=S_{2}=S$, so they have the same genus, that is, $1+2^{g_{1}}(g_{1}-1)=1+2^{g_{2}}(g_{2}-1)$. This asserts that $g_{1}=g_{2}=g$. In fact, if we assume $g_{1}>g_{2}$, then the above equality is equivalent to $1<2^{g_{1}-g_{2}}=(g_{2}-1)/(g_{1}-1)<1$, a contradiction. Now, this asserts that $S$ is a $(g,2)$-Fermat curve and it contains $(g,2)$-Fermat groups $H_{1}$ and $H_{2}$ such that $X_{j}=S/H_{j}$. It follows from part (3) of Theorem \ref{teo1} that $H_{1}=H_{2}$, that is, $X_{1}$ and $X_{2}$ are isomorphic.
\end{proof}

%%%%%%%%%%%%%%%%%%%%%%
%%%%%%%%%%%%%%%%%%%%%%
\section{An application to embeddings of moduli spaces}\label{Sec:moduli}

%%%%%%%%%%%%%%%%%%
\subsection{Teichm\"uller and moduli spaces of Fuchsian groups of the first kind}
Let $\Gamma$ be a finitely generated Fuchsian group of the first kind, that is, a discrete subgroup of the group ${\rm PSL}_{2}({\mathbb R})$ of conformal automorphisms of the upper half plane ${\mathbb H}^{2}$, whose limit set is all the extended real line. We will also assume that $\Gamma$ is not a triangular group, that is, $S_{\Gamma}={\mathbb H}^{2}/\Gamma$ is not an orbifold of genus zero with exactly three cone points (including punctures). 
By a {\it Fuchsian geometric representation} of $\Gamma$ we mean an injective homomorphism $\theta:\Gamma \hookrightarrow {\rm PSL}_{2}({\mathbb R}): a \mapsto \theta(a)=f \circ a \circ f^{-1}$, where $f:{\mathbb H}^{2} \to {\mathbb H}^{2}$ is a quasiconformal homeomorphism whose Beltrami coefficient $\mu \in L^{\infty}_{1}({\mathbb H}^{2})$ is compatible with $\Gamma$, that is, $\mu(a(z))\overline{a'(z)}=\mu(z)a'(z)$, for every $a \in \Gamma$ and a.e. $z \in {\mathbb H}^{2}$ (see \cite{Nag} for details). 
Two such Fuchsian geometric representations $\theta_{1}$ and $\theta_{2}$ are {\it Teichm\"uller equivalent} if there is some $A \in {\rm PSL}_{2}({\mathbb R})$ such that $\theta_{2}(a)=A \circ \theta_{1}(a) \circ A^{-1}$, for every $a \in \Gamma$. The set ${\mathcal T}(\Gamma)$, of those Teichm\"uller equivalence classes, is called the  {\it Teichm\"uller space}  of $\Gamma$. Let ${\mathbb L} \subset {\mathbb C}$ be the lowest half-plane and  $H^{2,0}(\Gamma)$ be the complex Banach space of all holomorphic maps $\psi:{\mathbb L} \to {\mathbb C}$ such that $\psi(a(z))a'(z)^{2}=\psi(z)$, for $a \in \Gamma$ and $z\in {\mathbb L}$, and $||\psi /{\rm Im}(z)^{2}||_{\infty}<\infty$. It is known the existence of an embedding (Bers embedding) $\rho:{\mathcal T}(\Gamma) \hookrightarrow H^{2,0}(\Gamma)$, with $\rho({\mathcal T}(\Gamma))$ being an open bounded contractible subset \cite{Bers, Fletcher-Markovic, Gardiner, Gardiner-Lakic,Nag}, in particular, providing a global holomorphic chart for ${\mathcal T}(\Gamma)$ and turn it into a simply connected Banach complex manifold, which is finite-dimensional.
If ${\mathbb H}^{2}/\Gamma$ is a surface of genus $g \geq 0$ with some number $r \geq 0$ of cone points, including punctures, then ${\mathcal T}(\Gamma)$ has dimension $3g-3+r$ (see, for instance, \cite{Nag}).

A Fuchsian geometric representation $\theta:\Gamma \hookrightarrow {\rm PSL}_{2}({\mathbb R})$, with $\theta(\Gamma)=\Gamma$, induces an 
automorphism $\rho \in {\rm Aut}(\Gamma)$, defined by $\rho(a)=\theta(a)$, called a {\it geometric automorphism} of $\Gamma$. Let us denote by ${\rm Aut}^{+}(\Gamma)$ the subgroup of ${\rm Aut}(\Gamma)$ formed by all the geometric automorphisms. Every $\rho \in {\rm Aut}(\Gamma)$ that 
preserves parabolic elements is, by Nielsens' theorem, of the form $\rho(a)=h \circ a \circ h^{-1}$, where $h:{\mathbb H}^{2} \to {\mathbb H}^{2}$ is some homeomorphism, which may or not preserve the orientation; so ${\rm Aut}^{+}(\Gamma)$ is an index two subgroup of ${\rm Aut}(\Gamma)$.

As the group ${\rm Inn}(\Gamma)$, of inner automorphisms of $\Gamma$, is a is a normal subgroup of ${\rm Aut}^{+}(\Gamma)$, we may consider 
${\rm Out}^{+}(\Gamma)={\rm Aut}(\Gamma)^{+}/{\rm Inn}(\Gamma)$, the group of {\it geometric exterior} automorphisms of $\Gamma$. There is natural action, by holomorphic automorphisms, of ${\rm Aut}^{+}(\Gamma)$ on ${\mathcal T}(\Gamma)$ defined by
${\rm Aut}^{+}(\Gamma) \times {\mathcal T}(\Gamma) \to {\mathcal T}(\Gamma): (\rho,[\theta]) \mapsto [\theta \circ \rho^{-1}].$
This action of ${\rm Aut}^{+}(\Gamma)$ is not faithful as, for $\rho \in {\rm Inn}(\Gamma)$, it holds that $\theta$ and $\theta \circ \rho^{-1}$ are Teichm\"uller equivalent. The induced action (again by holomorphic automorphisms)
${\rm Out}^{+}(\Gamma) \times {\mathcal T}(\Gamma) \to {\mathcal T}(\Gamma): (\rho,[\theta]) \mapsto [\theta \circ \rho^{-1}]$
turns out to be faithful (with the exception of few cases). Moreover,  ${\rm Out}^{+}(\Gamma)$ acts properly discontinuously on ${\mathcal T}(\Gamma)$. In \cite{Royden},  Royden proved that these are all the biholomorphisms of ${\mathcal T}(\Gamma)$ for $\Gamma$ torsion-free co-compact (i.e., $S_{\Gamma}$ is a closed Riemann surface) and later extended by Earle and Kra in \cite{EK} to the case that $\Gamma$ is finitely generated of type $(g,n)$ (i.e., $S_{\Gamma}$ is an analytically finite Riemann surface of genus $g$ and $r$ cone points) if $2g+r>4$.

 The quotient topological space ${\mathcal M}(\Gamma)={\mathcal T}(\Gamma)/{\rm Out}^{+}(\Gamma)$ is called the {\it moduli space} of $\Gamma$ (it is formed by all the ${\rm PSL}_{2}({\mathbb R})$-conjugacy classes of the Fuchsian groups $\theta(\Gamma)$, where $\theta$ runs over all Fuchsian representations of it, is a complex orbifold of the same dimension as ${\mathcal T}(\Gamma)$. 
 
%%%%%%%%%%%%%%%%%
\subsection{Embedding of moduli spaces}
Let $K$ be a finite index subgroup of $\Gamma$, different from the trivial one. Then $K$ is also finitely generated of the first kind. As every Fuchsian geometric representation of $\Gamma$ restricts to a Fuchsian geometric representation of $K$ and this restriction process respects the Teichm\"uller equivalence, there is a holomorphic embedding $\Theta_{K}:{\mathcal T}(\Gamma) \hookrightarrow {\mathcal T}(K)$. Let $\pi_{\Gamma}:{\mathcal T}(\Gamma) \to {\mathcal M}(\Gamma)$ and $\pi_{K}:{\mathcal T}(K) \to {\mathcal M}(K)$ be the corresponding projection holomorphic maps onto the moduli spaces. In general, there might not be a map $\Phi_{K}: {\mathcal M}(\Gamma) \to {\mathcal M}(K)$ such that $\pi_{K} \circ \Theta_{K}=\Phi_{K} \circ \pi_{\Gamma}$. 

We say that $K$ is a {\it geometrical characteristic subgroup} if it is invariant under the action of ${\rm Out}^{+}(\Gamma)$. For instance, if $K$ is a characteristic subgroup of $\Gamma$, then it is a geometrical one.

\begin{lemm}\label{lema1}
The existence of $\Phi_{K}: {\mathcal M}(\Gamma) \to {\mathcal M}(K)$ such that $\pi_{K} \circ \Theta_{K}=\Phi_{K} \circ \pi_{\Gamma}$ is given if 
$K$ is a geometrical characteristic subgroup of $K$.
\end{lemm}
\begin{proof}
Note that for the existence of $\Phi_{K}$ is enough to the following property. Let $\rho_{1}$ and $\rho_{2}$ be Fuchsian geometrical representations of $\Gamma$ which are ${\rm Out}^{+}(\Gamma)$-equivalent. Then we need to be sure that the restricted Fuchsian geometrical representations to $K$ are equivalent under ${\rm Out}^{+}(K)$. As we are assuming $K$ to be geometrical characteristic, every element of ${\rm Out}^{+}(\Gamma)$ induces (by restriction) an element of ${\rm Out}^{+}(K)$.
\end{proof}

So, from the above, if $K$ is a geometrical characteristic subgroup of $\Gamma$, then we have the induced map $\Phi_{K}$ (which is holomorphic if both $\Gamma$ and $K$ are finitely generated). Next, we describe the conditions for its injectivity.

\begin{lemm}\label{lema2}
Let $K$ be a geometrical characteristic subgroup of $\Gamma$. Then $\Phi_{K}$ is injective if and only if for every pair of Fuchsian geometrical representations
$\theta_{1}$ and $\theta_{2}$ of $\Gamma$, such that $\theta_{1}(K)=\theta_{2}(K)$, it holds that $\theta_{1}(\Gamma)$ and $\theta_{2}(\Gamma)$ are ${\rm PSL}_{2}({\mathbb R})$-conjugated.
\end{lemm}
\begin{proof}
Given $[\theta(K)] \in {\mathcal M}(K)$, the cardinality of $\Phi_{K}^{-1}([\theta(K)])$ is equal to the maximal number of Fuchsian geometric representations $\{\theta_{j}\}_{j \in J}$, such that $\theta_{j}(K)=\theta(K)$ and, for $j_{1} \neq j_{2}$,  $\theta_{j_{1}}(\Gamma)$ and $\theta_{j_{2}}(\Gamma)$ are not  ${\rm PSL}_{2}({\mathbb R})$-conjugated. 
\end{proof} 

%%%%%%%%%%%%%%%%%
\subsubsection{Example}
Let us consider a Fuchsian group $\Gamma \cong \langle \delta_{1},\ldots,\delta_{n+1}:\delta_{1}^{k}=\cdots=\delta_{n+1}^{k}=\prod_{j=1}^{n+1}\delta_{j}=1\rangle$, where  $(n-2)(k-2)>1$.
In this case, ${\mathbb H}^{2}/\Gamma$ is an orbifold of genus $g=0$ and with exactly $n+1$ cone points, each one of order $k$. 
The surface ${\mathbb H}^{2}/\Gamma'$ is of genus $g_{k,n}=1+k^{n-1}((n-1)(k-1)-2)/2$.

In \cite{HKLP} it was observed that $\Gamma'$ satisfies the conditions of Lemma \ref{lema2} (in fact, the much stronger result was proved in the above paper:  ``{\it If \; $\Gamma'_{1}=\Gamma'_{2}$, then $\Gamma_{1}=\Gamma_{2}$}").  As a consequence, $\Phi_{\Gamma'}$ provides a holomorphic embedding ${\mathcal M}_{0,n+1}={\mathcal M}(\Gamma) \subset {\mathcal M}(\Gamma')={\mathcal M}_{g_{k,n}}$.

Note that, if $n=3$ and $k \geq 4$, then ${\mathcal T}(\Gamma)={\mathbb H}^{2}$ and the 
above asserts that the Teichm\"uller disc  $\Theta_{\Gamma'}({\mathcal T}(\Gamma)) \subset {\mathcal T}(\Gamma')$ projects under $\pi_{\Gamma'}$ to a genus zero one-punctured curve (that is, a copy of the complex plane) in the moduli space ${\mathcal M}(\Gamma')={\mathcal M}_{g_{k,3}}$ (i.e. an example of a Teichm\"uller curve).

%%%%%%%%%%%%%%%%
\subsection{The case of $(g,k)$-Fermat curves}
Let us consider a co-compact torsion-free Fuchsian group $\Gamma$ of genus $g \geq 2$ and let $k \geq 2$.  The surface $S={\mathbb H}^{2}/\Gamma^{(k)}$ is a $(g,k)$-Fermat curve and $H=\Gamma/\Gamma^{(k)}$ a $(g,k)$-Fermat group. As $\Gamma^{(k)}$ is a characteristic subgroup of $\Gamma$, Lemma \ref{lema1} asserts the existence of a holomorphic map
$\Phi_{\Gamma^{(k)}}:{\mathcal M}(\Gamma) \to {\mathcal M}(\Gamma^{(k)})$.

\begin{prop}\label{propo3}
The injectivity of $\Phi_{\Gamma^{(k)}}:{\mathcal M}(\Gamma) \to {\mathcal M}(\Gamma^{(k)})$ is equivalent for every $(g,k)$-Fermat curve to have a unique, up to conjugation by conformal automorphisms, $(g,k)$-Fermat group. 
\end{prop}
\begin{proof}
By Lemma \ref{lema2},  $\Phi_{\Gamma_{k}}$ is non-injective if and only if there are two Fuchsian geometrical representations $\theta_{1}, \theta_{2}$ of $\Gamma \cong \pi_{g}$ such that $\theta_{1}(\Gamma^{(k)})=\theta_{2}(\Gamma^{(k)})=K_{0}$ and with $\theta_{1}(\Gamma)$ and $\theta_{2}(\Gamma)$ being not ${\rm PSL}_{2}({\mathbb R})$ conjugated, that is, the $(g,k)$-Fermat curve $S={\mathbb H}^{2}/K_{0}$ admits two non-conjugated $(g,k)$-Fermat groups $H_{1}=\theta_{1}(\Gamma)/K_{0}$ and $H_{2}=\theta_{2}(\Gamma)/K_{0}$.
\end{proof}

As a consequence of Theorem \ref{teo1} together with Proposition \ref{propo3}, we obtain the following fact.

\begin{prop}\label{teo3}
Let $\Gamma$ be a co-compact torsion free Fuchsian group of genus $g \geq 2$, $r \geq 1$ an integer, and let $p \geq 3$ be a prime integer such that $g\notin \{1+ap+b(p-1)/2; a,b \in \{0,1,\ldots\}\}$.
Then $\Phi_{\Gamma^{(p^{r})}}:{\mathcal M}(\Gamma) \to {\mathcal M}(\Gamma^{(p^{r})})$ is injective.
\end{prop}

%%%%%%%%%%%%%%%%%%
%%%%%%%%%%%%%%%%%%

\end{document}